\documentclass{article}
\usepackage[T1]{fontenc}
\usepackage[pdftex]{color,graphicx}
\usepackage[numbers,square]{natbib}
\usepackage{amsmath,amssymb,amsthm,eucal,upref}
\date{}
\theoremstyle{plain}
\newtheorem{theorem}{Theorem}[section]
\newtheorem{lemma}[theorem]{Lemma}
\newtheorem{proposition}[theorem]{Proposition}

\theoremstyle{definition}

\theoremstyle{definition}

\numberwithin{equation}{section}

\newcommand{\RR}{\mathrm{I\kern-0.20emR}}
\newcommand{\E}{\mathrm{e}\kern0.2pt}%

\newcommand{\be}{\begin{equation}}
\newcommand{\ee}{\end{equation}}

\begin{document}

\title{{\bf N-modal steady water waves with vorticity}}

\author{Vladimir Kozlov$^a$, Evgeniy Lokharu$^a$}

\date{}

\maketitle

\vspace{-10mm}

\begin{center}
$^a${\it Department of Mathematics, Link\"oping University, S--581 83 Link\"oping}\\

\vspace{1mm}

E-mail: vlkoz@mai.liu.se / V.~Kozlov; evgeniy.lokharu@liu.se / E.~Lokharu \\

\end{center}

\begin{abstract}

Two-dimensional steady gravity driven water waves with vorticity are considered. Using a multidimensional bifurcation argument, we prove the existence of small-amplitude periodic steady waves with an arbitrary number of crests per period. The role of bifurcation parameters is played by the roots of the dispersion equation.

\end{abstract}

\section{Introduction}

The existence theory for rotational steady water waves of finite depth goes back to the paper \cite{DJ} of Dubreil-Jacotin published in 1934. However only after the paper \cite{CS} of Constantin and Strauss appeared in 2004 the theory of rotational steady waves has attracted a lot of attention from the mathematical community. The latter paper was devoted to large amplitude Stokes waves with an arbitrary vorticity distribution. Stokes waves are periodic travelling waves with exactly one crest and one trough in every minimal period. A natural question can be raised here: are there periodic waves with more complicated geometry? To simplify the discussion we will restrict this question to the case of small-amplitude periodic waves without surface tension. In the class of periodic uniderectional waves with vorticity only Stokes waves exist (see \cite{GW}, \cite{KKL_BL} and \cite{KK_Disp}). Thus, to answer the question affirmatively one shall consider a wider class of waves allowing internal stagnation or critical layers. The first positive result in this direction was obtained (see \cite{En2}) in 2011 by Ehrnstr\"om, Escher and Wahl\'en . They prove existence of bimodal travelling waves of small amplitude (see also \cite{War}). A bimodal wave has several crests per period and the first approximation of the surface profile is given by a combination of two simple modes: cosine functions with different wavelengths. Later, in 2015 Ehrnstr\"om and Wahl\'en proved (see \cite{En3}) the existence of trimodal waves for which the first approximation is given by a combination of three basic modes. Both papers use a similar technique based on a bifurcation argument. The crucial role there is the choice of the bifurcation parameters. The main difficulty for the higher order bifurcation is the fact that the number of natural parameters of the problem is limited. Normally, one may consider only the Bernoulli constant or the wavelength as a bifurcation parameters. In the case of a linear vorticity, the derivative of the vorticity function (which is a constant in this case) is another possible parameter. The main novelty of our approach is that an arbitrary number of the bifurcation parameters are used. These parameters are the roots of the dispersion equation while the corresponding vorticity is a perturbation of a linear function. Our argument is different to the one used in \cite{En3} to construct trimodal waves. In \cite{En3} the authors use a Lyapunov-Schmidt reduction directly to the initial infinite-dimensional problem. This complicates calculation of the determinant of the matrix of the reduced system which is essential for the proof of the existence. In contrast, we first reduce the problem to a finite dimensional system for which the linear part is given by an ordinary differential operators and then use Lyapunov-Schmidt reduction. After that the corresponding matrix is diagonal and then we can quite straightforward prove the existence of symmetric and periodic waves with an arbitrary number of basic modes.

The paper is organized as follows. In the beginning of Sect. 2 we formulate the problem in a suitable way. In Sect. 2.1 and 2.2 we discuss stream solutions and the linear approximation for the initial problem. The Sect. 3 is an essential part of the proof and concerns to the inverse spectral problem related to the dispersion equation. Here we prove Theorem \ref{th:ISP} allowing to use roots of the dispersion equation as bifurcation parameters. Next, in Sect. 4, we first formulate our main result. The rest of the paper is dedicated to the proof of the main theorem. We start by writing the problem in an operator form and then in Sect. 4.4 we reduce it to a finite-dimensional system of equations. The rest of the proof is contained in Sect. 4.5, where we perform Lyapunov-Schmidt reduction to the reduced system. In contrast to the classical bifurcation theory, we find solutions only for the values of small parameters $t = (t_1,...,t_N)$ such that $|t|^2 < \epsilon |t_j|$ for some small $\epsilon$.

\section{Statement of the Problem}

Let an open channel of a uniform rectangular cross-section be bounded below by a
horizontal rigid bottom and let water occupying the channel be bounded above by a
free surface not touching the bottom. In an appropriate Cartesian coordinates $(x,y)$,
the bottom coincides with the $x$-axis and gravity acts in the negative
$y$-direction. The steady water motion is supposed to be two-dimensional and rotational; the
surface tension is neglected on the free surface of the water, where the pressure is
constant. These assumptions and the fact that the water is incompressible allow us to
seek the velocity field in the form $(\psi_y, -\psi_x)$, where $\psi (x,y)$ is
referred to as the {\it stream function}. The vorticity distribution $\omega$ is
supposed to be a prescribed smooth function depending only on the values of $\psi$.

We choose the frame of reference so that the velocity field is time-independent as
well as the unknown free-surface profile. The latter is assumed to be the graph of a function
$y = \eta (x)$, $x \in \Bbb R$, where $\eta$ is a positive continuous function, and
so the longitudinal section of the water domain is ${\cal D} = \{ x \in \Bbb R , \ 0 < y <
\eta (x) \}$. We use the non-dimensional variables proposed by Keady and Norbury \cite{KN}. Namely, lengths and velocities are scaled to $(Q^2/g)^{1/3}$ and $(Qg)^{1/3}$ respectively. Here $Q$ and $g$ are the dimensional quantities for the rate of flow and the gravity
acceleration respectively, whereas $(Q^2/g)^{1/3}$ is the depth of the critical
uniform stream in the irrotational case. This scaling leads to to an equivalent problem, provided the mass flux is not zero.

The following free-boundary problem for $\psi$ and $\eta$ has long been known (cf. \cite{KN}):
\begin{eqnarray}
&& \psi_{xx} + \psi_{yy} + \omega (\psi) = 0, \quad (x,y) \in {\cal D} ; \label{eq:lapp} \\
&& \psi (x,0) = 0, \quad x \in \Bbb R ; \label{eq:bcp} \\ && \psi (x,\eta (x)) = 1,
\quad x \in \Bbb R ; \label{eq:kcp} \\ && |\nabla \psi (x,\eta (x))|^2 + 2 \eta (x) = 3
r, \quad x \in \Bbb R . \label{eq:bep}
\end{eqnarray}
In the condition \eqref{eq:bep} (Bernoulli's equation), $r$ is a constant considered as
a problem's parameter and referred to as Bernoulli's constant/the total head. The problem \eqref{eq:lapp}-\eqref{eq:bep} describes two-dimensional steady water waves with a nonzero mass flux. In what follows, we will assume that $\omega \in C^{\infty}(\RR)$.

For the further analysis of the problem it is convenient to rectify the domain $\cal D$ by scaling the vertical variable to
\[
z = y \frac{d}{\eta (x)},
\]
while the horizontal coordinate remains unchanged. The parameter $d>0$, which by meaning is the depth of a uniform stream to be determined later, can be chosen arbitrarily. In other words, for a given $d>0$ we will find a vorticity function and a stream solution for which we prove the existence of multi-modal waves of small amplitude bifurcating from the uniform stream.

Thus, the domain $\cal D$ transforms into the strip $S = \RR \times (0,d)$. Next, we introduce a new unknown function $\hat{\Phi} (x, z)$ on
$\bar S$ by
\[
\hat{\Phi} (x, z) = \psi \left( x, \frac{z}{d} \, \eta (x) \right). \label{zeta/phi}
\]
A direct calculation shows that the problem \eqref{eq:lapp}-\eqref{eq:kcp} reads in new variables as
\begin{align}
& \left[ \hat{\Phi}_x - \frac{z \eta_x}{\eta} \hat{\Phi}_z \right]_{x} - \frac{z \eta_x}{\eta} \left[ \hat{\Phi}_x - \frac{z \eta_x}{\eta} \hat{\Phi}_z \right]_z + \left( \frac{d}{\eta} \right)^2 \hat{\Phi}_{zz} + \omega(\hat{\Phi}) = 0; \label{lapp1} \\
& \hat{\Phi}(x,0) = 0, \ \ \ \hat{\Phi}(x,d) = 1, \ \ x \in \RR; \label{boc1}
\end{align}
while the Bernoulli equation \eqref{eq:bep} becomes
\begin{equation} \label{bep1}
\hat{\Phi}_z^2 = \frac{\eta^2}{d^2} \left( \frac{3r - 2 \eta}{1 + \eta_x^2} \right).
\end{equation}
In what follows by a solution of \eqref{lapp1}-\eqref{bep1} we mean a pair $(\hat{\Phi},\eta) \in C^{2,\alpha}(\bar{S})\times C^{1,\alpha}(\RR)$ satisfying these equations.
%$\hat{\Phi} \in C^{2,\alpha}(\overline{S})$, $\eta \in C^{1,%\alpha}(\RR)$ and $\omega \in C^{\infty}(\RR)$, where $\alpha \in %(0,1)$ is fixed and remains unchanged throughout the paper.

\subsection{Stream solutions}

A pair $(u(y), d)$, where $u:[0,d] \to \RR$ and $d>0$ is called a stream solution corresponding to the vorticity function $\omega$ and the depth $d$, if
\begin{equation} \label{streamsol}
u'' + \omega(u) = 0, \ \ u(0) = 0, \ \ u(d) = 1.
\end{equation}
The corresponding Bernoulli constant $r$ is calculated by
\[
3r = [u'(d)]^2 + 2d.
\]
Note that in our definition of a stream solution the Bernoulli constant is not fixed.
%Let us assume that the vorticity $\omega$ is fixed. Then one may look for a family of all solutions of \eqref{streamsol} for different values of parameters $d$ and $r$. In general, there exist constant $r_c$ and $d_0$ such that $r \geq r_c$ and $d \leq d_0$ for any solution $(u,d)$ of \eqref{streamsol}.

In this paper we will be interested in the case when the vorticity $\omega$ is a small perturbation of a linear vorticity $\omega_0(p) = bp$, where $b$ is a positive constant:
\begin{equation} \label{vort}
\omega(p) = \omega_{\delta}(p) = b p + \delta_1 \omega_1(p) + ... + \delta_N \omega_N(p).
\end{equation}
Here the coefficients $\delta_j$ are small and compactly supported functions $\omega_j \in C^{\infty}_0(0,1)$ will be chosen later. The choice of the sign of the constant $b$ is crucial because negative values of $b$ give rise only to unidirectional uniform streams, for which the dispersion equation has at most one root and all periodic small-amplitude waves are Stokes waves.

We will require the constants $b$ and $d$ to satisfy
\begin{equation} \label{vortcond}
\sqrt{b} \neq \frac{\pi j}{2d}, \ \ j=1,2,....
\end{equation}
The assumption \eqref{vortcond} guarantees that the problem \eqref{streamsol} with the vorticity $\omega$ of the form \eqref{vort} possesses a unique solution $u$ with $u'(d) \neq 0$, provided all $\delta_j$ are small enough. To prove that \eqref{streamsol} has a unique solution with $\omega$ given by \eqref{vort}, we write $u = u_0 + v$ and consider the problem
\[
v'' + b v + \sum_{j=1}^N \delta_j \omega_j(u_0 + v) = 0, \ \ v(0) = v(d) = 0.
\]
We can rewrite this problem in an operator form as follows:
\[
Lv + Q(v,\delta) = 0,
\]
where the operators
\[
L,Q:\{ u \in C^{2,\alpha}([0,d]): u(0) = u(d) = 0 \} \to C^{\alpha}([0,d])
\]
are defined by $Lv = v'' + bv$ and $Q(v, \delta) = \sum_{j=1}^N \delta_j \omega_j(u_0 + v)$. The assumption \eqref{vortcond} guarantees that operator $L$ has a trivial kernel, which ensures it's invertibility. This allows us to apply implicit function theorem in order to find a family of stream solutions $u(z;\delta)$ in a neighbourhood of $\delta = 0$ which depends analytically on $\delta$ (see \cite{KP} for a good overview). Note that \eqref{vortcond} also implies that $u_y(d; \delta) \neq 0$ for all small $\delta$, which ensures that there is no stagnation on the surface.  %The sign of $u_y(d)$ plays no role in our proof and one can show that multi-modal waves appear for both positive and negative values of $u_y(d)$.

\subsection{Linear approximation of the water-wave problem}

Let us consider a linear approximation of the problem. For this purpose we formally linearize equations \eqref{lapp1}-\eqref{bep1} near some stream solution $(u,d)$.  Thus, we put
\[
\hat{\Phi} = u + \epsilon \Phi + O(\epsilon^2), \ \ \eta = d + \epsilon \zeta + O(\epsilon^2).
\]
Using this ansatz in \eqref{lapp1}-\eqref{bep1}, we find
\begin{align*}
& \left[ \Phi_x - \frac{z \zeta_x u_z}{d} \right]_x + \Phi_{zz} - \frac{2 \zeta u_{zz}}{d} + \omega'(u) \Phi = O(\epsilon), \\
& \Phi(x,0) = \Phi(x,d) = 0,\\
& \Phi_z\vert_{z=d} - \left( \frac{u'(d)}{d} - \frac{1}{u'(d)}  \right) \zeta = O(\epsilon).
\end{align*}
We can simplify equations by letting
\[
\Psi = \Phi - \frac{z \zeta u_z}{d}.
\]
The latter transformation was used in \cite{En2} and \cite{En3}. The formula above implies $\zeta = -\Psi\vert_{z=d} / u'(d)$ and then the linear part of the above problem transforms into
\begin{align}
& \Psi_{xx}+\Psi_{zz} + \omega'(u) \Psi = 0 \nonumber \\
& \Psi\vert_{z=0} = 0 \label{lin} \\
& \Psi_z\vert_{z=d} - \kappa \Psi\vert_{x=d} = 0, \nonumber
\end{align}
where  $\kappa = 1/[u'(d)]^2 - \omega(1)/u'(d)$. Separation of variables in the system \eqref{lin} leads to the following Sturm-Liouville problem:
\begin{equation} \label{SLdisp}
 - \phi_{zz}  - \omega' (u)  \phi = \mu \phi \ \mbox{on} \ (0, d), \quad
\phi (0) = 0 , \ \  \, \phi_z (d) = \kappa \phi (d).
\end{equation}
The spectrum of \eqref{SLdisp} depends only on the triple $(\omega, d, u)$ and consists of a countable set of simple real eigenvalues $\{ \mu_j \}_{j=1}^{\infty}$ ordered so that $\mu_j < \mu_l$ for all $j < l$. Furthermore, only a finite number of eigenvalues may be negative.  The normalized eigenfunction corresponding to an eigenvalue $\mu_j$ will be denoted by $\phi_j$. Thus, the set of all eigenfunctions $\{\phi_j\}_{j=1}^\infty$ forms an orthonormal basis in $L^2(0,d)$.

Solving the linear problem \eqref{lin}, we find that the space of bounded and even in the $x$-variable solutions is finite-dimensional and is spanned by the functions
\[
\Psi(x,z) = \cos(\sqrt{|\mu_j|} x) \phi_j(z),
\]
where $\mu_j \leq 0$.

%The eigenvalue problem \eqref{SLdisp} arises in a several papers. In \cite{En1} authors study the inverse problem and prove that the latter problem admits any number of negative eigenvalues for some vorticity distribution and a proper stream solution. The paper \cite{En3} is devoted to the case of linear vorticity for which the linear problem has a three dimensional kernel. Multidimensional kernels for a linear vorticity have been studied in \cite{War}. On the other hand, in \cite{KK_Disp} authors consider an arbitrary vorticity distributions in \eqref{SLdisp}.

Let us turn to the problem \eqref{SLdisp} for the linear vorticity $\omega_0(p)=bp$.

\begin{proposition} \label{PropLin} For any $N \geq 1$ and $d>0$ there exists $b > 0$ satisfying \eqref{vortcond} such that the Sturm-Liouville problem \eqref{SLdisp} for the triple $(\omega_0, d, u_0)$ has exactly $N$ negative eigenvalues while all other eigenvalues are positive.
\end{proposition}

In the formulation above $u_0$ stands for the unique solution of \eqref{streamsol} for the pair $(\omega_0,d)$. An analog of Proposition \ref{PropLin} was proved in \cite{War}.

\begin{proof}[Proof of Proposition \ref{PropLin}] For a given $d > 0$, we put
\[
b = \left( \frac{\pi N}{d} \right)^2 - \epsilon
\]
for some $\epsilon \neq 0$, which can be positive or negative. Note that \eqref{vortcond} is satisfied automatically for small $\epsilon$. Now we can calculate
\[
\kappa = \frac{1}{b} \frac{\sin^2 \sqrt{b}d}{\cos^2 \sqrt{b}d} - \sqrt{b} \frac{\sin \sqrt{b}d}{\cos \sqrt{b}d} = \frac{\epsilon d}{2\pi N} + O(\epsilon^2), \ \ \epsilon \to 0.
\]
For any $j>1$ the eigenfunction corresponding to $\mu_j$ is given by $C\sin{ \sqrt{b+\mu_j}z}$ and then the boundary relation at $z = d$ in \eqref{SLdisp} gives
\[
\sqrt{b+\mu_j} \cos \sqrt{b+\mu_j} d = \kappa \sin \sqrt{b+\mu_j} d,
\]
which implies that eigenvalues $\mu_j$ are uniformly separated from zero for all $j > 1$, provided $\epsilon$ is small. More precisely, there is a positive constant $\beta > 0$ such that $|\mu_j| > \beta$ for all $j>1$ and all $|\epsilon| < \beta$. It is known that eigenvalues $\mu_j$ alternates with eigenvalues $\lambda_j^D$ of the Dirichlet problem:
\[
 - \phi_{zz}  - b \phi = \lambda^D \phi \ \mbox{on} \ (0, d), \quad
\phi (0) = \phi(d) = 0.
\]
More precisely, we have $\mu_1 < \lambda_1^D$ and for any two neighbour eigenvalues $\lambda_i^D < \lambda_{i+1}^D$ there is exactly one eigenvalue $\mu_j \in (\lambda_i^D,\lambda_{i+1}^D)$. Finally, the choice of the constant $b$ implies $\lambda_N^D = \epsilon$. Thus, there are exactly $N-1$ to the left of $\lambda_{N-1}^D$ and $\mu_N < \epsilon$. Furthermore, we necessary have $|\mu_N| < \beta$, which implies that $\mu_N < -\beta$ and $\mu_{N+1} > \beta$, provided $|\epsilon|< \beta$. Thus all eigenvalues $\mu_1,...,\mu_N$ are negative, while all others are positive. This finishes the proof.
\end{proof}

Note that it follows from the proof of the proposition that the constant $\kappa$ in \eqref{SLdisp} has the same sign as $\epsilon$ and then can be both positive or negative. The sign of $\kappa$ plays no role in our proof and one can show that multi-modal waves appear for both positive and negative values of $\kappa$.

\section{Inverse spectral problem}

For a given $N > 1$ and $d>0$, we consider the constant $b$ provided by Proposition \ref{PropLin} and let $(u_0,d)$ be the stream solution \eqref{streamsol}, which corresponds to the vorticity function $\omega_0(p) = bp$ and the depth $d$. According to the Proposition \ref{PropLin}, the Sturm-Liouville problem \eqref{SLdisp} for the triple $(\omega_0, d, u_0)$ has exactly $N$ negative eigenvalues $\lambda_1<...<\lambda_N < 0$ and all other eigenvalues are positive.

Now, let the functions $\omega_k$, $k=1,\ldots,N$, in (\ref{vort}) be fixed. Then for small $\delta=(\delta_1,\ldots,\delta_N)$ there exists a unique solution $u_\delta$ to the problem (\ref{streamsol}) with $\omega=\omega_\delta$ given by (\ref{vort}) and we can consider eigenvalue problem \eqref{SLdisp} for the triple $(\omega_\delta, d, u_\delta)$. Let $\mu_1<\ldots<\mu_N$ be all negative eigenvalues of that Sturm-Liouville problem, while $\phi_j$, $j>N$ are eigenfunctions corresponding to negative eigenvalues respectively. Note that both eigenvalues $
\mu_j \in \RR$ and the corresponding eigenfunctions $\phi_j \in L^2(0,d)$ depend analytically on $\delta$ (for details see Chapter VII in \cite{Kato} and examples 1.15 and 2.12 therein). Let us define the map
\begin{equation} \label{K1}
T:\delta=(\delta_1,\ldots,\delta_N)\rightarrow\mu=(\mu_1,\ldots,\mu_N).
\end{equation}
which is analytic in a neighborhood of $\delta=0$. The next theorem provides the invertibility of $T$ and by this we solve the spectral inverse problem locally: for a given set of eigenvalues located near $\lambda_1,...,\lambda_N$, we find a corresponding potential through the vorticity distribution $\omega_\delta$ and the stream solution $(u_{\delta},d)$.
\begin{theorem} \label{th:ISP} There exist functions $\omega_k\in C^\infty_0(0,1)$, $k=1,\ldots,N$, such that the Jacobian matrix of the map {\rm (\ref{K1})}
is invertible at $\delta=0$.
\end{theorem}
\begin{proof} Let us  calculate the Jacobian matrix of the map $T$. Since $u_\delta$ depends analytically on $\delta$ (as an operator function with values in $C^{2,\alpha}([0,d])$) as it was noted in Subsection 2.1, we have
\[
u_\delta(z) = u_0(z) + \delta_1 u_1(z) + ... + \delta_N u_N(z) + O(|\delta|^2), \ \ |\delta| \to 0,
\]
where $u_1(0) = u_1(d) = ... = u_N(0) = u_N(d) = 0$. Using this ansatz, we find that every $u_j$ is subject to
\[
u_j'' + b u_j = -\omega_j(u_0), \ \ u_j(0) = u_j(d) = 0
\]
for all $1 \leq j \leq N$. Solving these equations, we find
\begin{equation} \label{u_j}
u_j(z) = c_j \frac{\sin \sqrt b z}{\sqrt b } - \frac 1 {\sqrt b} \int_0^z \omega_j(u_0(p)) \sin \sqrt b (z-p) dp,
\end{equation}
where
\[
c_j = \frac 1 { \sin (\sqrt b d)} \int_0^d \omega_j(u_0(p)) \sin \sqrt b (d-p) dp.
\]
Let us turn to the dispersion equation for the triple ($\omega_\delta, d, u_\delta$) which is given by the following eigenvalue problem
\begin{align*}
- & \phi'' - \omega'(u_\delta) \phi = \mu \phi \ \ \text{on} \ \ (0,d); \\
& \phi'(d) = \left( \frac 1 { k_\delta^2 } - \frac {\omega_\delta(1)}{k_\delta} \right) \phi(d), \ \ \ \phi(0) = 0.
\end{align*}
Here $k = u_\delta'(d)$. The corresponding eigenfunctions $\phi_l$ and eigenvalues $\mu_l$, $l = 1,...,N$ depend smoothly on $\delta$, so that
\[
\mu_l = \lambda_l + \delta_1 \mu_{l1} + ... \delta_N \mu_{lN} + O(|\delta|^2), \ \ |\delta| \to 0, \ \ l = 1,...,N.
\]
In order to find the coefficients $\mu_{lj}$, we write the eigenvalue problem above in operator form:
\[
A \phi = \mu \phi,
\]
where the self-adjoint operator $A$ is defined by
\[
\langle A \phi, \psi \rangle = -\left( \frac 1 { k_\delta^2 } - \frac {\omega_\delta(1)}{k_\delta} \right) \psi(d) \phi (d) + \int_0^d \phi' \psi' dz - \int_0^d \omega_\delta'(u) \phi \psi dz,
\]
where $\phi, \psi \in W^{1,2}_0(0,d)$ and the last space consists of functions from $W^{1,2}(0,d)$ vanishing at $0$. Here $\langle \cdot, \cdot \rangle$ stands for the standard inner product in $L^2(0,d)$. Thus, if
\[
A = A_0 + \delta_1 A_1 + ... + \delta_N A_N + O(|\bar \delta|^2), \ \ |\bar \delta| \to 0,
\]
where operators $A_j$ are defined by
\[
\langle A_j \phi, \psi \rangle =  \frac{u_j'(d)}{k_0^2} \left( \frac 2 {k_0} - b \right)   \phi(d) \psi(d) - \int_0^d \omega_j'(u_0) \phi \psi dz
\]
for $\phi, \psi \in W^{1,2}_0(0,d)$. Using \eqref{u_j}, we find
\begin{equation} \label{u_j'}
u_j'(d) = -\frac{1}{\sin(\sqrt{b}d)} \int_0^d \omega_j(u_0) \sin\sqrt{b}z dz=-\int_0^d\omega_j'(u_0)\frac{\cos^2{\sqrt{b}z}}{\sin^2{\sqrt{b}d}}dz
\end{equation}
after integration by parts and observation that $u_0=\sin{\sqrt{b}z}/\sin{\sqrt{b}d}$.
Then, after some  algebra one obtains
\[
\mu_{lj} =  \langle A_j \phi_l, \phi_l \rangle,
\]
where $\phi_l$ stands for the normalized eigenfunction corresponding to the eigenvalue $\lambda_l$ of the Sturm-Liouville problem for $\delta=0$. Therefore, using (\ref{u_j'}), we conclude
\begin{equation} \label{mulj}
\mu_{lj}  =  \int_0^d \omega_j'(u_0)(A_l\cos^2{\sqrt{b}z}-\phi_l^2) dz,\;\;A_l=\frac{\phi_l^2(d)}{k_0^2\sin^2{\sqrt{b}d}}\Big(b-\frac{2}{k_0}\Big).
\end{equation}
%\begin{equation} \label{muljaaa}
%\mu_{lj}  = \left[ \frac{u_j'(d)}{k^2} \left( \frac 2 {k} - b \right) + \frac{\omega_j(1)}{k} \right] \phi_l(d) \phi_l(d) - \int_0^d \omega_j'(u_0) \phi_l^2 %dz.
%\end{equation}

Let us transform the right-hand side in \eqref{mulj}. Since $u_0=\sin\sqrt{b}z/\sin\sqrt{b}d$, we can reduce integration over $[0,d]$ to a smaller interval of where the function $u_0$ is monotone. We put
\[
\Lambda = \frac{2\pi}{\sqrt{b}}
\]
which is the period of the function $u_0$.

Consider first the case $\sin\sqrt{b}d>0$. Let  $d_*$ be the smallest positive root of the equation $u_0(d_*)=1$. Then the function $u_0$ is monotone on the interval $[0,d_*]$. Let $z\in (0,d_*)$. To describe all  roots of $u_0(\hat{z})=u_0(z)$ on the interval $(0,d)$, which are different from $z$, we introduce the integer $M$ as the largest integer satisfying
\begin{equation}\label{K2a}
\Lambda\Big(M+\frac{3}{4}\Big)<d.
\end{equation}
Then the roots of $u_0(\hat{z})=u_0(z)$   are given by
\begin{equation}\label{K2b}
z_k^\pm=\Big(k+\frac{3}{4}\Big)\Lambda\pm\Big(z+\frac{\Lambda}{4}\Big),\;\;k=1,\ldots,M.
\end{equation}
Now we can write (\ref{mulj}) as
\begin{equation}\label{K2}
\mu_{lj}  =  \int_0^{d_*} \omega_j'(u_0)\left[(2M+1)A_l\cos^2{\sqrt{b}z}-\phi_l^2(z)-\sum_{k=1}^M(\phi_l^2(z_k^+)+\phi_l^2(z_k^-)) \right] dz.
\end{equation}

Note that the function $\phi_l$ solves the problem
\[
-\phi_l'' - b \phi_l = \lambda_l \phi_l, \ \ \phi_l(0) = 0, \ \ \phi_l'(d) = \kappa_0 \phi_l(d).
\]
Hence, depending on the sign of $b+\lambda_l$, it is either $C_l \sin{ \sqrt{b + \lambda_l} z }$ or $C_2 \sinh{ \sqrt{b + \lambda_l} z }$. Let us consider the case when $b+\lambda_1 > 0$ and hence
\[
\phi_l(z) = C_l \sin{ \sqrt{b + \lambda_l} z }\;\;\mbox{for all $l=1,\ldots,N$}.
\]
Then
\begin{eqnarray*}
&&\phi_l^2(z)+\sum_{k=1}^M(\phi_l^2(z_k^+)+\phi_l^2(z_k^-))\\
&&=C_l^2\left((M+1/2)-\cos (2\sqrt{b+\lambda_l}z)-B_l\cos(2\sqrt{b+\lambda_l}(z+\Lambda/4))\right),
\end{eqnarray*}
where
$$
B_l=\sum_{k=1}^M\cos\big(2\sqrt{b+\lambda_l}(k+3/4)\Lambda\big).
$$
Therefore
\begin{equation}\label{K3}
\mu_{lj}  =  \int_0^{d_*} \omega_j'(u_0)f_l(z)dz,
\end{equation}
where
\begin{eqnarray}\label{K4}
&&f_l(z)=(M+1/2)\Big(A_l-C_l^2+A_l\cos(2\sqrt{b}z)\Big)\\
&&+C_l^2\left(\cos (2\sqrt{b+\lambda_l}z)+B_l\cos(2\sqrt{b+\lambda_l}(z+\Lambda/4)\right).\nonumber
\end{eqnarray}
%$$
%f_l(z)=(2M+1)A_l\cos^2{\sqrt{b}z}-C_l^2\left((M+1/2)-\cos (2\sqrt{v+\lambda_l}z)-B_l\cos(2\sqrt{b+\lambda_l}(z+\Lambda/4)\right).
%$$
Since the functions
$$
1,\;\cos(2\sqrt{b}z),\;\cos(2\sqrt{b+\lambda_l}z),\;\cos(2\sqrt{b+\lambda_l}(z+\Lambda/4)),\;l=1,\ldots,N,
$$
are linear independent, the set of functions
\begin{equation}\label{K5}
\cos(\sqrt{b}z),\;\;f_l(z),\;l=1,\ldots,N,
\end{equation}
is also linear independent. Indeed, the only case when they can be dependent is $\sqrt{b}=2\sqrt{b+\lambda_l}$ for certain $l$, but then either $A_l$ or $A_l - C_l^2$ is non-zero in the representation (\ref{K4}) for $f_l$ which is sufficient for linear independence of the system (\ref{K5}).

In the case $b+\lambda_1\leq 0$ we have $b+\lambda_2>0$. Otherwise we will have two eigenfunctions corresponding to different eigenvalues and which are no orthogonal to each other. Now we must replace the function $\sin\sqrt{b+\lambda_1}$ by $\sinh\sqrt{b+\lambda_1}$ and, repeating then the above argument, we obtain the linear independence of the system (\ref{K5}).

Using linear independence of the functions (\ref{K5}), we can find functions $\alpha_j=\alpha_j(p)$ in $C_0^\infty(0,1)$ such that
\begin{equation}\label{K6}
\int_0^{d_*}\alpha_j(u_0(z))f_l(z)dz=\delta_l^j\;\;\mbox{and}\;\;\int_0^{d_*}\alpha_j(u_0(z))\cos(\sqrt{b}z)dz=0,
\end{equation}
where $\delta_l^j$ is the Kronecker delta. If we put
$$
\omega_j(p)=\int_0^p\alpha_j(s)ds=\int_0^Z\alpha_j(u_0(z))u'_0(z)dz,\;\;p=u_0(Z),
$$
then we see that $\omega_j\in C_0^\infty(0,1)$ due to the second equality in (\ref{K6}) and $\omega_j'(p)=\alpha_j(p)$. Therefore the matrix
(\ref{K3}) is invertible by the first equality in (\ref{K6}). This completes the proof in the case $\sin\sqrt{b}d>0$.

Assume now that $\sin\sqrt{b}d<0$. Then the function $u_0$ is negative on $(0,\Lambda/2)$ and it is monotonically increasing from $0$ to $1$ on the interval
$[\Lambda/2,d_*+\Lambda/2]$. Reasoning as above we obtain the representation (\ref{K3}) but on the interval $(\Lambda/2,d_*+\Lambda/2)$ with $3/4$ replaced by $1/4$ in formulas like (\ref{K2a}) and (\ref{K2b}). The remaining part of the proof in this case is the same as above.

\end{proof}

A similar inverse problem was considered in \cite{En1}, where the authors used inverse spectrum Sturm-Liouville theory (see \cite{Dl} for details). Unfortunately we could not implement a similar argument based on the inverse Sturm-Liouville theory and so we used another approach based on a perturbation argument.

\section{Existence of $N$-modal waves}

For a given $N>1$ and $d>0$ let $b$ be the constant from Proposition \ref{PropLin} so that the Sturm--Liouville problem (\ref{SLdisp}) has $N$ negative eigenvalues $\lambda_1,\ldots\lambda_N$ and all remaining eigenvalues are positive.
By Theorem \ref{th:ISP} we can choose real-valued functions $\omega_1,\ldots,\omega_N\in C_0^\infty(0,1)$  such that the map (\ref{K1}) is invertible in
 a certain neighborhood $\Gamma\subset\Bbb R^N$ of the point $(\lambda_1,\ldots,\lambda_N)$. Furthermore, we assume that for every $(\mu_1,...,\mu_N) \in \Gamma$ we have $\mu_j<0$, $j=1,...,N.$ In what follows we will use $\mu=(\mu_1,\ldots,\mu_N)\in\Gamma$
as a parameter of our water wave problem and denote by $\omega(p;\mu), u(z;\mu)$ the perturbed vorticity $\omega_{\delta(\mu)}$ and stream solution $u_{\delta(\mu)}$. We choose also $\mu^*=(\mu_1^*,\ldots,\mu_N^*)\in\Gamma$ so that $\mu_j^*/\mu_l^* $ are rational numbers for all $1\leq j,l \leq N$. In this case all solutions to the linear problem \eqref{lin} for the triple $(\omega(p; \mu_*), d, u(z; \mu_*))$ have a common period which we denote by $\Lambda_*$. Since the set $\Gamma$ is open, we can always find some $\mu^*$ with this property.

The aim of the present paper is to give an affirmative answer to the question raised in \cite{En2}: does higher-order bifurcation occur so that $N$-modal waves exist for all $N \geq 1$?

\begin{theorem} \label{Th1} For any integer $N\geq 1$ and $d>0$ there exist constants $\beta,\epsilon > 0$, a linear vorticity distribution $\omega_0(p) = bp$ with $b>0$ and a smooth function $\mu:[-\beta,\beta]^N \to \Gamma$ with the following property: for any $t:=(t_1,...,t_N) \in \RR^N$ such that $|t| < \beta$ and $|t|^2/|t_j| \leq \epsilon$ the nonlinear water wave problem \eqref{lapp1}-\eqref{bep1} with the vorticity $\omega(p; \mu(t))$ possesses a unique $\Lambda_*$-periodic and even solution $(\Psi, \eta)$ corresponding to the Bernoulli constant $r(t) = [u_z(d; \mu(t))^2 + 2d]/3$ such that
\[
\eta(x) = d + t_1 \cos(\sqrt{|\mu_1^*|}x) + ... + t_N \cos(\sqrt{|\mu_N^*|}x) + O(|t|^2).
\]
The constants $\beta$ and $\delta$ depend only on $N, d, \mu_*$ and $\Lambda_*$.
\end{theorem}

The waves constructed in Theorem \ref{Th1} are symmetric with respect to the vertical line $x=0$ and this is essential for the proof of the theorem. The only existence result for non-symmetric two-dimensional waves in the absence of surface tension is \cite{KLsym}, while there are several numerical results: \cite{Z1, Z2, WV}. There are also some results on two-dimensional non-symmetric gravity-capillary steady water waves (see \cite{BGT}).

\subsection{Functional-analytic setup}

In order to write equations \eqref{lapp1}-\eqref{bep1} in an operator form, we define nonlinear operators
\begin{align*}
 \hat{\cal F}_1 (\hat{\Phi},\eta; \mu) = & \left[ \hat{\Phi}_x - \frac{z \eta_x}{\eta} \hat{\Phi}_z \right]_{x} - \frac{z \eta_x}{\eta} \left[ \hat{\Phi}_x - \frac{z \eta_x}{\eta} \hat{\Phi}_z \right]_z + \\
 & \left( \frac{d}{\eta} \right)^2 \hat{\Phi}_{zz} + \omega(\hat{\Phi}; \mu), \\
\hat{\cal F}_2 (\hat{\Phi},\eta; \mu) = & {\hat{\Phi}_z^2 \vert}_{z = d} - \frac{\eta^2}{d^2} \left( \frac{3r(\mu) - 2 \eta}{1 + \eta_x^2} \right),
\end{align*}
where
\[
r(\mu) = [u'(d;\mu)]^2 - 2 d.
\]
The definitions above imply that
\[
\hat{\cal F}_1(u(\cdot; \mu),d; \mu) = \hat{\cal F}_2(u(\cdot; \mu),d; \mu) = 0
\]
for all $\mu \in \Gamma$.
Next, we put
\[
\Phi(x,z) = \hat{\Phi}(x,z) - u(z; \mu) - \frac{z u_z(z; \mu) (\eta(x) - d)}{d}.
\]
The purposes of this change of variables are twofold. First, it gives a formal linearization near the stream solution $u(z; \mu)$. Second, it allows to eliminate the profile $\eta$ from the equations. Indeed, the definition above implies that
\begin{equation} \label{eta}
\eta(x) = d - \frac{{\Phi \vert}_{z=d}}{u'(d)}.
\end{equation}
Thus, we naturally define
\[
{\cal F}_j(\Phi; \mu) = \hat{\cal F}_j \left(\Phi + u + \frac{zu_z(\eta-d)}{d}, d - \frac{{\Phi \vert}_{z=d}}{u_z(d)}; \mu\right), \ \ j = 1,2.
\]
Therefore, the problem \eqref{lapp1} - \eqref{bep1} with the vorticity $\omega(\cdot; \mu)$ and Bernoulli constant $r(\mu)$ reads as
\begin{equation} \label{Op}
{\cal F}(\Phi; \mu) = 0,
\end{equation}
where ${\cal F} = ({\cal F}_1,{\cal F}_2):X \times \Gamma \to Y:= Y_1 \times Y_2$ and the spaces are defined by
\[
\begin{split}
X = \{\Phi \in C^{2,\alpha}_{per}(\bar{S}):\ \Phi(x,0) = 0, \ \Phi(x,z) = \Phi(-x,z) \ \ \text{for all} \ \ (x,y) \in S \}
\end{split}
\]
and
\[
Y_1 = C^{0, \alpha}_{per}(\bar{S}), \ \ Y_2 = C^{1, \alpha}_{per}(\RR).
\]
Here and elsewhere the subscript $per$ denotes $\Lambda_*$-periodicity and evenness in the horizontal $x$-variable. As the norms in these spaces we will use $\|\cdot\|_{C^{0,\alpha}([-\Lambda_*/2,\Lambda_*/2] \times [0,d])}$ and $\|\cdot\|_{C^{1,\alpha}([-\Lambda_*/2,\Lambda_*/2])}$ respectively.
%Our goal is to prove the existence of a bifurcation sheet of solutions for $\mu$ near $\mu_*$.

\begin{proposition} The Fr\'echet derivative $D_{\Phi}{\cal F}$ at $\Phi = 0$ is given by
\begin{align}
& [D_{\Phi}{\cal F}_1(0,\mu)](\Phi) = \Phi_{xx} + \Phi_{zz} + \omega'(u)\Phi \label{lin1} \\
& [D_{\Phi}{\cal F}_2(0,\mu)](\Phi) = \Phi_z\vert_{z=d} - \kappa \Phi \vert_{z=d}, \label{lin2}
\end{align}
where
\[
\kappa = \frac{1}{[u_z(d; \mu)]^2} - \frac{\omega(1;\mu)}{u_z(d,\mu)}.
\]
\end{proposition}
\begin{proof}
The statement follows from a direct calculation.
\end{proof}

\subsection{Spectral decomposition}

According to our notations $\mu = (\mu_1,...,\mu_N)$ stands for the first $N$ eigenvalues of the Sturm-Liouville problem \eqref{SLdisp} for the triple $(\omega(\cdot,\mu), d, u(\cdot,\mu))$. Note that by construction these eigenvalues are negative, while all others are positive. Let
\[
\phi_1(z;\mu),...,\phi_N(z; \mu)
\]
be the corresponding eigenfunctions. We define projectors
\[
P\Phi = \sum_{j=1}^N \Phi_j \phi_j, \ \ \ \widetilde{P} = {\rm id} - P.
\]
Here $\Phi_j = \langle \Phi, \phi_j \rangle$ and $\langle \cdot, \cdot \rangle$ stands for the standard scalar product in $L^2(0,d)$. The projectors $P$ and $\widetilde{P}$ are orthogonal projectors in $L^2(0,d)$ and are well defined operators on $X$. Note that the complete set of eigenfunctions $\{\phi_j\}_{j \in \mathbb{N}}$ is a basis in the spaces
\[
H^n_0(0,d) = \{ \phi \in H^n(0,d): \ \ \phi(0) = 0 \}
\]
for $n=0$ and $n=1$. In general the eigenfunction are orthogonal with respect to the following bilinear form
\[
{\cal B}(\phi,\psi) = \int_0^d [\phi' \psi' - \omega'(u) \phi \psi] dz - \kappa \phi(d) \psi(d),
\]
which is well defined on $H^1_0(0,d) \times H^1_0(0,d)$. Integrating by parts, we find that
\[
{\cal B}(\phi,\phi_j) = \mu_j \langle \phi, \phi_j \rangle, \ \ j \in \mathbb{N}, \ \ \phi \in C^{2,\alpha}_0([0,d]).
\] 
In what follows we will often use the following identity
\begin{equation} \label{Bid}
\begin{split}
\int_0^d (-\phi_{zz} - \omega'(u(z))) \phi) \phi_j dz & = -[\phi_z(d)-\kappa \phi(d)] \phi_j(d) + {\cal B}(\phi,\phi_j)\\
&= -[\phi_z(d)-\kappa \phi(d)] \phi_j(d) + \mu_j \langle \phi, \phi_j \rangle,
\end{split}
\end{equation}
which is valid for any function $\phi \in C^{2,\alpha}_0([0,d])$ and $j \geq 1$.

Let us write equation \eqref{Op} in the following form, where linear and nonlinear parts of the operator are separated:
\begin{align}
& \Phi_{xx} + \Phi_{zz} + \omega'(u)\Phi = N_1(\Phi;\mu) \label{N1} \\
& \Phi_z\vert_{z=d} - \kappa \Phi \vert_{z=d} = N_2(\Phi; \mu) \label{N2},
\end{align}
where $N_1$ and $N_2$ are nonlinear parts of the operators ${\cal F}_1$ and ${\cal F}_2$ respectively. Thus, multiplying \eqref{N1} by $\phi_j$, $j=1,...,N$ and using \eqref{Bid}, we obtain equations for projections:
\begin{align}
 & \Phi_j'' - \mu_j \Phi_j = \langle N_1, \phi_j \rangle  - N_2 \phi_j(d) \label{Phij}, \ \ \ j=1,...,N;
\end{align}
and
\begin{align}
& \widetilde{\Phi}_{xx} + \widetilde{\Phi}_{zz} + \omega'(u) \widetilde{\Phi} = \widetilde{P}(N_1) + N_2 \sum_{j=1}^N \phi_j(d) \phi_j(z) \label{tild1} \\
& \widetilde{\Phi}_z(x,d)  - \kappa \widetilde{\Phi}(x,d) = N_2 \label{tild2}
\end{align}
where $\Phi_j = \langle \Phi, \phi_j \rangle$ and $\widetilde{\Phi} = \widetilde{P}\Phi$. It is clear that if some functions $\Phi_j$ and $\widetilde{\Phi} \in \textrm{ran}(\textrm{id} - P)$ solve equations \eqref{Phij}-\eqref{tild2}, then the function
\[
\Phi(x,z) = \sum_{j=1}^N \Phi_j(x) \phi_j(z) + \widetilde{\Phi}(x,z)
\]
solves \eqref{Op}.

We will show below that the function $\widetilde{\Phi}$ can be resolved from equations \eqref{tild1}-\eqref{tild2} as an operator of $(\Phi_1,...,\Phi_N)$ and $\mu$ in a neighbourhood of the origin.

We note that  an asymptotic analysis of solutions to ordinary differential equations with operator coefficients based on spectral decomposition was developed by Kozlov and Maz'ya in \cite{KM1} for linear problems and in \cite{KM2} for non-linear ones.

\subsection{The reduction to a finite-dimensional system}

First, we consider the linear part of \eqref{tild1}-\eqref{tild2} which is given by
\begin{align}
& \widetilde{\Phi}_{xx} + \widetilde{\Phi}_{zz} + \omega'(u) \widetilde{\Phi} = f \label{lin1} \\
& \widetilde{\Phi}_z(x,d)  - \kappa \widetilde{\Phi}(x,d) = g \label{lin2}\\
& \widetilde{\Phi}(x,0) = 0. \label{lin3}
\end{align}
Let $\widetilde{X}$ be the subspace of $X$ which consists of functions $\Phi$ satisfying
\[
\widetilde{P}[\Phi(x,\cdot)] = \Phi(x,\cdot)
\]
for all $x \in \RR$. Furthermore, we define the range space $\widetilde{Y}$ to be a subspace of $Y$ which consists of all $(f,g) \in Y$ satisfying
\[
\langle f(x,\cdot), \phi_j(\cdot) \rangle = g(x) \phi_j(d)
\]
for all $x \in \RR$. Let us define a linear operator $L:=(L_1,L_2):\widetilde{X} \to \widetilde{Y}$ by
\begin{align*}
& L_1 \widetilde{\Phi} = \widetilde{\Phi}_{xx} + \widetilde{\Phi}_{zz} + \omega'(u) \widetilde{\Phi} \\
& L_2 \widetilde{\Phi} = \widetilde{\Phi}_z(x,d)  - \kappa \widetilde{\Phi}(x,d).
\end{align*}
Now we are ready to prove

\begin{lemma} \label{iso} The mapping $L:\widetilde{X} \to \widetilde{Y}$ is a linear isomorphism and its norm depends only on $b, d$ and $\Lambda_*$.
\end{lemma}
\begin{proof} First, we note that operator $L$ is bounded and $L(\widetilde{X}) = \widetilde{Y}$. To verify sujectivity it is enough to consider the right-hand side in \eqref{lin1}-\eqref{lin2} given by basic Fourier modes:
	\[
	f(x,z) = \hat{f}_{n}(z) \cos(n \tau_* x), \ \ g(x) = \hat{g}_{n} \cos(n \tau_* x),
	\]
	where $\tau_* = 2\pi \Lambda_*^{-1}$, $n \in \mathbb{N}\cup \{0\}$ and $\hat{f}_{n} \in C^{\alpha}([0,d])$. Furthermore, the relation
	\begin{equation} \label{frel}
	\langle \hat{f}_{n}, \phi_j \rangle = \hat{g}_{n} \phi_j(d), \ \ \ j=1,...,N
	\end{equation}
	is true. Thus, problem \eqref{lin1}-\eqref{lin3} is equivalent to
\begin{align*}
& \phi'' + (-n^2 \tau_*^2 + \omega'(u)) \phi = \hat{f}_{n} \\
& \phi_z(d) - \kappa \phi(d) =  \hat{g}_{n} \\
& \phi(0) = 0,
\end{align*}	
where $\phi \in \widetilde{P}(C^{2,\alpha}([0,d]))$. Now expanding $\phi = \sum_{j=1}^{\infty}c_j \phi_j$, one finds
\[
\langle \phi, \phi_j \rangle = (-n^2\tau^2_*-\mu_j)^{-1} (\langle \hat{f}_{n}, \phi_j \rangle - \hat{g}_{n} \phi_j(d)).
\]
Thus, we find a solution $\phi \in C^{2,\alpha}([0,d])$. Furthermore, because of \eqref{frel}, we have $\langle \phi, \phi_j \rangle = 0$ for all $1 \leq j \leq N $ so that $\phi \in \widetilde{P}(C^{2,\alpha}([0,d]))$. \\

We have shown that $L$ is surjective. It is left to show that $L$ is one-to-one mapping. To prove that  we need to obtain estimates of the solution by means of the right-hand side. For this purpose we apply Schauder estimate (see Theorem 7.3 \cite{ADN}) to the system \eqref{lin1}-\eqref{lin3} and find
\begin{equation} \label{L2}
\|\widetilde{\Phi}\|_{C^{2,\alpha}(R_{\Lambda_*})} \leq C [\|f\|_{Y_1} + \|g\|_{Y_2} + \|\widetilde{\Phi}\|_{L^2(R_{\Lambda_*})}],
\end{equation}
where $R_{\Lambda_*} = [-\Lambda_*/2,\Lambda_*/2] \times [0,d]$.  Let us show that
\begin{equation} \label{L1}
\|\widetilde{\Phi}\|_{L^2(R_{\Lambda_*})} \leq C^* [\|f\|_{Y_1} + \|g\|_{Y_2}].
\end{equation}
In order to prove \eqref{L1}, we multiply equation \eqref{lin1} by $\widetilde{\Phi}$ and integrate over $z \in [0,d]$, which gives
\begin{equation}\label{lemma_L:1}
\int_0^d \widetilde{\Phi}_{xx} \widetilde{\Phi} dz + \int_0^d [\widetilde{\Phi}_{zz} + \omega'(u)\widetilde{\Phi}] \widetilde{\Phi} dz = \int_0^d f \widetilde{\Phi} dz.
\end{equation}

To calculate the second integral on the left-hand side, we use \eqref{Bid} to show that
\[
\begin{split}
-\int_0^d [\widetilde{\Phi}_{zz} + \omega'(u)\widetilde{\Phi}] \phi_j dz & = - \phi_j(d) g(x) + \mu_j \int_0^d \widetilde{\Phi} \phi_j dz.
\end{split}
\]
Because the series
\[
\widetilde{\Phi} = \sum_{j=N+1}^{+\infty} \langle \widetilde{\Phi}, \phi_j \rangle \phi_j
\]
converges uniformly on $[0,d]$ (see Theorem 2.27, \cite{AlGw}), we find
\begin{equation}\label{lemma_L:2}
\begin{split}
 -\int_0^d [\widetilde{\Phi}_{zz} + \omega'(u)\widetilde{\Phi}] \widetilde{\Phi} dz  =
   & \sum_{j = N+1}^{\infty} \mu_k [\langle \widetilde{\Phi}, \phi_j \rangle ]^2 - \widetilde{\Phi}(x,d) g \geq \\
& \mu_{N+1} \int_0^d \widetilde{\Phi}^2 dz - \widetilde{\Phi}(x,d) g.
\end{split}
\end{equation}

Here $\{\mu_j\}_{j=1}^{\infty}$ stands for the set of all eigenvalues of the Sturm-Liouville problem \eqref{SLdisp} for the triple $(\omega(\cdot,\mu), d, u(\cdot,\mu))$. We recall that $\mu_j > 0$ for all $j > N$.

Now we integrate \eqref{lemma_L:1} over $R_{\Lambda_*}$. After integration by parts, we get
\[
\iint_{R_{\Lambda_*}} \widetilde{\Phi}_x^2 dx dz + \left[ - \iint_{R_{\Lambda_*}} [\widetilde{\Phi}_{zz} + \omega'(u)\widetilde{\Phi}] \widetilde{\Phi} dx dz  \right] \leq
\iint_{R_{\Lambda_*}} |f \widetilde{\Phi}|  dx dz.
\]
Combining this inequality with \eqref{lemma_L:2}, we arrive at
\[
 \iint_{R_{\Lambda_*}} \widetilde{\Phi}^2 dx dz \leq
 \mu_{N+1}^{-1} \left[ \iint_{R_{\Lambda_*}} |f \widetilde{\Phi}|  dx dz + \int_{[-\Lambda_*/2, \Lambda_*/2 ]} |\widetilde{\Phi}(x,d) g| dx \right].
\]
For an arbitrary $\beta > 0$, we write
\[
\iint_{R_{\Lambda_*}} |f \widetilde{\Phi}|  dx dz \leq \beta \iint_{R_{\Lambda_*}} \widetilde{\Phi}^2 dx dz + \beta^{-1} \iint_{R_{\Lambda_*}} f^2 dx dz
\]
and, similarly,
\[
\begin{split}
\int_{[-\Lambda_*/2, \Lambda_*/2 ]} |\widetilde{\Phi}(x,d) g|  dx & \leq \iint_{R_{\Lambda_*}} |\widetilde{\Phi}_z(x,z) g|  dx dz  \\
& \leq \beta \Lambda_* d \| \widetilde{\Phi} \|_{C^{2,\alpha}(R_{\Lambda_*})}^2 + \beta^{-1} \| g \|^2_{L^2([-\Lambda_*/2, \Lambda_*/2 ])}.
\end{split}
\]
Composing the last three inequalities, we obtain
\[
\|\widetilde{\Phi}\|_{L^2(R_{\Lambda_*})}^2 \leq \frac{2 \beta \Lambda_* d}{\mu_{N+1}} \|\widetilde{\Phi}\|^2_{X} + \frac{\Lambda_* d}{\beta \mu_{N+1}} \|f\|^2_{Y_1} + \frac{\Lambda_* d}{\beta \mu_{N+1}} \|g\|^2_{Y_2}.
\]
To complete the proof of the lemma it is left to combine this inequality with \eqref{L2}, provided
\[
\beta = \frac{\mu_{N+1}}{C 4 \Lambda_* d},
\]
where $C$ is the constant from \eqref{L2}.
\end{proof}

Before formulating the next theorem, let us define the space
\[
X_{red} = [C^{2,\alpha}_{per}(\RR)]^N,
\]
where the space $C^{2,\alpha}_{per}(\RR)$ consists of all even and $\Lambda_*$-periodical functions from $C^{2,\alpha}(\RR)$.

\begin{theorem} \label{Thred} For any $m\geq 1$ there exist open neighborhoods of the origin $U \subset X_{red}$, $V \subset \widetilde{X}$ and a smooth operator $h \in C^m(U; V)$ such that the function $\widetilde{\Phi}:=h(\Phi_1,...,\Phi_N)$ is the unique solution of the boundary problem \eqref{tild1}-\eqref{tild2}, provided $(\Phi_1,...,\Phi_N) \in U$.
\end{theorem}
\begin{proof} The statement follows directly from Lemma \ref{iso} and the implicit function theorem applied to the system \eqref{tild1}-\eqref{tild2}.
\end{proof}

Theorem \ref{Thred} allows to reduce the problem for small-amplitude waves to a finite dimensional system of the form
\begin{equation} \label{reduced}
\Phi_j'' - \mu_j \Phi_j = N_j(\bar{\Phi}; \mu),
\end{equation}
where $\bar{\Phi} = (\Phi_1,...,\Phi_N)$ and $\Phi_j \in C^{2,\alpha}_{per}(\RR)$, $j=1,...,N$. The nonlinear operators $N_j$ are smooth and
\[
D_{\bar{\Phi}}N_j(0; \mu) = 0.
\]

\subsection{The Lyapunov-Schmidt reduction and the existence of $N$-modal waves}

The kernel of the linear operator on the left-hand side in \eqref{reduced} for $\mu = \mu^*$ is spanned by
\[
\xi_1=\begin{pmatrix}
\cos(\sqrt{|\mu_1^*|}x) \\
0 \\ \vdots \\ 0
\end{pmatrix},
\ \xi_2=\begin{pmatrix} 0 \\
\cos(\sqrt{|\mu_2^*|}x) \\
 \vdots \\ 0
\end{pmatrix},...,
\ \xi_N=\begin{pmatrix}
0 \\ \vdots \\ 0 \\
\cos(\sqrt{|\mu_N^*|}x)
\end{pmatrix}.
\]
Let us write
\begin{equation} \label{ans}
\Phi_j = t_j \cos(\sqrt{|\mu_j^*|}x) + \zeta_j,
\end{equation}
where $\zeta_j$ is orthogonal to $\cos(\sqrt{|\mu_j^*|}x)$ in $L^2(-\Lambda_*/2, \Lambda_*/2)$. We note that the vector function $(\zeta_1,...,\zeta_N)$ is orthogonal to all $\xi_j, j=1,...,N$, and hence this splitting corresponds to the Lyapunov-Schmidt decomposition under the action of the projection $\bar{\Phi} = (\Phi_1,...,\Phi_N)$ to the space $X_N$ spanned by $\xi_1,...,\xi_N$.

Thus, substituting \eqref{ans} into \eqref{reduced} and taking the projection on $X_N$ and its compliment, we arrive at the equations
\begin{align}
& \zeta_j''-\mu_j^* \zeta_j  =  \zeta_j (\mu_j - \mu_j^*) + N^*_j(t,\zeta; \mu) - G_j(t, \zeta; \mu) \cos(\sqrt{|\mu_j^*|}x) , \label{red1} \\
& t_j (\mu_j - \mu_j^*) = G_j(t, \zeta; \mu). \label{red2}
\end{align}
The nonlinear parts are defined by
\[
N^*_j(t,\zeta; \mu) = N_j(\bar{\Phi}; \mu), \ \ G_j(t, \zeta; \mu) = \frac{2}{\Lambda_*}\int_{-\Lambda_*/2}^{\Lambda_*/2} N_j(\bar{\Phi}; \mu) \cos(\sqrt{|\mu_j^*|}x) dx,
\]
where $\bar{\Phi} = (\Phi_1,...,\Phi_N)$ depends on $t$ and $\zeta$ via \eqref{ans}. The linear operator on the left-hand side in \eqref{red1} defined on the space $X_*^N$, where
\[
X_* = \{g \in C^{2,\alpha}_{per}(\RR): \int_{-\Lambda_*/2}^{\Lambda_*/2} g(x) \sin{\sqrt{|\mu_j^*|x}} dx = 0, \ \ j=1,...,N \}
\]
with values in
\[
\{g \in C^{\alpha}_{per}(\RR): \int_{-\Lambda_*/2}^{\Lambda_*/2} g(x) \sin{\sqrt{|\mu_j^*|x}} dx = 0, \ \ j=1,...,N \}
\]
is invertible, while the right hand side is presented by some terms of order $|t|^2$ plus small perturbations of $\zeta$. Thus, using implicit function theorem, we can resolve $\zeta = (\zeta_1,...,\zeta_N)$ from \eqref{red1} as an operator of $t = (t_1,...,t_N)$ and $\mu = (\mu_1,...,\mu_N)$, provided $|t|$ and $|\mu - \mu^*|$ are small enough. Furthermore, the following estimate is valid:
\[
\|\zeta_j\|_{C^{2,\alpha}(-\Lambda_*/2, \Lambda_*/2)} \leq C |t|^2.
\]
 Therefore, system \eqref{red1}-\eqref{red2} is reduced to a system of scalar equations:
\begin{equation} \label{tmu}
t_j (\mu_j - \mu_j^*) = G_j(t, \zeta(t,\mu) \mu),
\end{equation}
where the nonlinear term satisfies
\begin{equation} \label{gj}
|G_j(t, \mu)| \leq C |t|^2, \ \ \left|\frac{\partial G_j }{\partial \mu} (t, \mu)\right| \leq C |t|^2
\end{equation}
for all sufficiently small $|t|$ and $|\mu|$, while the constant $C$ is independent of $\mu$. Let us rewrite \eqref{tmu} as
\[
\mu_j = \mu_j^* + G_j(t,\mu) / t_j.
\]
Thus, if $|t|^2/|t_j| \leq 1/(2C)$, where $C$ is the constant from \eqref{gj}, we can apply fixed point theorem to resolve $\mu$ as a function of $t$ so that
\[
\mu_j = \mu_j^* + O(|t|^2/\epsilon).
\]
Thus, tracking back all changes of variables, we find a solution $(\Phi,\mu)$ to \eqref{Op} such that
\[
\Phi(x,z; \mu) = \sum_{j=1}^N t_j \cos(\sqrt{|\mu_j^*|} x) \phi_j(z) + O(|t|^2).
\]
Now the formula \eqref{eta} recovers the profile as
\[
\eta(x) = d - \sum_{j=1}^N \frac{t_j}{u'(d)} \cos(\sqrt{|\mu_j^*|} x) \phi_j(z) + O(|t|^2).
\]
Using these identities, we can express the stream function via
\[
\begin{split}
\psi(x,y) &= u(yd/\eta(x)) \\
&+ \sum_{j=1}^N t_j \cos(\sqrt{|\mu_j^*|} x) \left[ \phi_j(yd/\eta(x)) - \frac{yd u_y(yd/\eta(x))}{\eta(x)}u_y(d)) \phi_j(d) \right]  \\
&+ O(|t|^2).
\end{split}
\]
This completes the proof of the theorem. \\

\textbf{Acknowledgements.} The authors are thankful to the anonymous referee for the help in improving the article. V. K. acknowledges the support of the Swedish Research Council (VR) grant EO418401.

\end{document}